\documentclass{amsart}
\usepackage[latin1]{inputenc}
\usepackage{amsfonts,amssymb,amsmath,amsthm}
\usepackage{rawfonts}
\usepackage{fourier}
\usepackage{tikz}
\newcommand{\BD}{\mbox{\rm bd}}
\newtheorem{theorem}{Theorem}
\newtheorem{lemma}[theorem]{Lemma}
\newtheorem{proposition}[theorem]{Proposition}
\newtheorem{fact}[theorem]{Fact}
\theoremstyle{definition}

\begin{document}

%%%%%%%%%%%%%%%%%%%%%%%%%%%%%%%%%%%%%%%%%%%%%%%%%%

\title{The complete characterization of tangram pentagons}

\author{Sarah Sophie Pohl}
\address{Ringau 40, 37327 Leinefelde, Germany}
\email{sarah.sophie.pohl@gmail.com}
\author{Christian Richter}
\address{Friedrich Schiller University, Institute for Mathematics, 07737 Jena, Germany}
\email{christian.richter@uni-jena.de}

\date{\today}

\begin{abstract}
The old Chinese puzzle tangram gives rise to serious mathematical problems when one asks for all tangram figures that satisfy particular geometric properties. All $13$ convex tangram figures are known since 1942. They include the only triangular and all six quadrangular tangram figures. The families of all $n$-gonal tangram figures with $n \ge 6$ are either infinite or empty.
Here we characterize all $53$ pentagonal tangram figures, including $51$ non-convex pentagons and $31$ pentagons whose vertices are not contained in the same orthogonal lattice. 
\end{abstract}

\dedicatory{To our teachers and friends Prof.\ Dr.\ Johannes B\"ohm on the occasion of his 95th birthday, Prof.\ Dr.\ Eike Hertel on the occasion of his 80th birthday and Dr.\ Carsten M\"uller for 30 years of service as the school master of the Carl Zeiss Gymnasium Jena.}

\maketitle

%%%%%%%%%%%%%%%%%%%%%%%%%%%%%%%%%%%%%%%%%%%%%%%%%%

\section{Introduction}

The tangram, known as a Chinese puzzle \cite{goodrich1817}, is a collection of seven polygons, called \emph{tans}: five isosceles right triangles, two with legs of length $1$, one with $\sqrt{2}$ and two with $2$, a square with sides of length $1$ and a parallelogram with sides of length $1$ and $\sqrt{2}$ and an angle of $\frac{\pi}{4}$. These seven pieces are arranged, using Euclidean isometries, to form dissections of prescribed or unknown polygons, as in Figure~\ref{fig:pieces}.
\begin{figure}
\begin{center}
\begin{tikzpicture}[xscale=.8, yscale=.8]
\draw[densely dotted] 
  (-1,-2)--(-1,2) (0,-2)--(0,2) (1,-2)--(1,2) (2,-2)--(2,2) (-1,-2)--(2,-2) (-1,-1)--(2,-1) (-1,0)--(2,0) (-1,1)--(2,1) (-1,2)--(2,2)
  ;
\draw[thick]
  (-1,-2)--(0,-2)--(2,0)--(0,2)--(-1,2)--cycle
	(0,2)--(0,-2)
	(-1,1)--(0,1)--(-1,0)--(0,-1)--(0,-2)--(-1,-1)
	(0,0)--(2,0)
	;

\draw[densely dotted] 
  (4,-2)--(6,-2) (4,-1)--(5,-1) (4,0)--(5,0) (4,1)--(6,1) (4,2)--(6,2) (4,-2)--(4,2) (5,-2)--(5,2)
	(6,-.586)--(8.121,1.535) (6,.828)--(7.414,2.242) (6,-.586)--(6.707,-1.297) (6.707,1.535)--(8.121,.121) (7.414,2.242)--(8.828,.828)
  ;
\draw[thick]
  (4,0)--(6,-2)--(8.828,.828)--(6,.828)--(6,2)--cycle
	(6,.828)--(7.414,-.586)
	(5,0)--(6,1)--(6,-1)--(5,-1)--(5,0)--(6,0)
	(6,-2)--(6,-1)
	(5,0)--(5,1)
	;
\end{tikzpicture}
\end{center}
\caption{
The seven tans along with their lattices, dissecting  a convex and a non-convex pentagon.\label{fig:pieces}}
\end{figure}
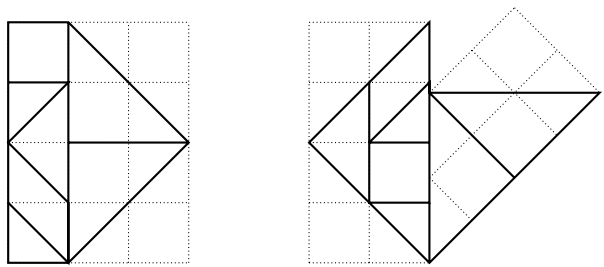
A \emph{dissection (or tiling) of a polygon $P$ into pieces (or tiles) $P_1,\ldots,P_k$} is given if $P$ is the union of all pieces $P_i$, $1 \le i \le k$, and if no two pieces have interior points in common. A polygon $T$ is called a \emph{tangram} if $T$ can be dissected into (isometric images of) the seven tans.

Tangram puzzles usually ask to find dissections of prescribed polygons \cite{elffers,goodrich1817,read1985,slocum}. We do not address aspects of the tangram related to craft, art and design. Fruitful mathematical problems appear when one aims to detect systematically all tangrams satisfying particular geometric properties. Such questions have been posed and studied in several books (e.g.\ \cite{elffers, mueller2020, read1985, slocum}, \cite[Chapter 7]{mueller2013}), papers in mathematical journals (e.g.\ \cite{gardner1974,graber2016, heinert1998, heinert1998a, read2004, wang}), private publications \cite{mueller2007} and contributions to mathematical competitions \cite{brunner2014,brunner2015,heinert1996,pohl2018,pohl2019,pohl2020}. The most prominent result of that kind is the following one by Wang and Hsiung (cf.\ Figure~\ref{fig:convex}).

\begin{theorem}[\cite{wang}]\label{thm:convex}
There exist, up to isometry, exactly $13$ convex tangrams: one triangle, six quadrangles, two pentagons and four hexagons.
\end{theorem}

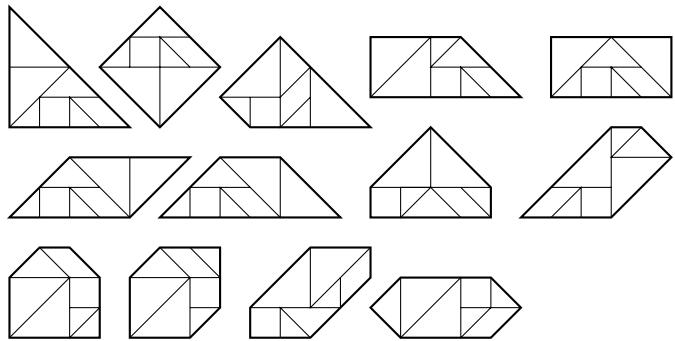
\begin{figure}
\begin{center}
\begin{tikzpicture}[xscale=.4, yscale=.4]
%triangle
\draw
  (0,0) node[above right] {\tikz[xscale=.4, yscale=.4]
	{
\draw[thick]
  (0,0)--(4,0)--(0,4)--cycle
	;
\draw
  (0,2)--(2,2)--(0,0) (1,0)--(1,1)--(3,1) (2,0)--(2,1)--(3,0)
  ;
	}}
	;
%quadrangle1
\draw
  (3,0) node[above right] {\tikz[xscale=.4, yscale=.4]
	{
\draw[thick]
  (0,2)--(2,0)--(4,2)--(2,4)--cycle
	;
\draw
  (0,2)--(1,2) (3,3)--(1,3)--(1,2)--(4,2) (3,2)--(2,3)--(2,0)
  ;
	}}
	;
%quadrangle2
\draw
  (7,0) node[above right] {\tikz[xscale=.4, yscale=.4]
	{
\draw[thick]
  (0,1)--(1,0)--(5,0)--(2,3)--cycle
	;
\draw
  (3,1)--(2,0)--(2,1)--(3,2)--(3,0) (2,3)--(2,1)--(0,1) (1,1)--(1,0)
  ;
	}}
	;
%quadrangle3
\draw
  (12,1) node[above right] {\tikz[xscale=.4, yscale=.4]
	{
\draw[thick]
  (0,0)--(5,0)--(3,2)--(0,2)--cycle
	;
\draw
  (0,0)--(2,2)--(2,0) (3,2)--(2,1)--(4,1) (3,0)--(3,1)--(4,0)
  ;
	}}
	;
%quadrangle4
\draw
  (18,1) node[above right] {\tikz[xscale=.4, yscale=.4]
	{
\draw[thick]
  (0,0)--(4,0)--(4,2)--(0,2)--cycle
	;
\draw
  (0,0)--(1,1)--(1,0) (4,0)--(2,2)--(1,1)--(3,1) (2,0)--(2,1)--(3,0)
  ;
	}}
	;
%quadrangle5
\draw
  (0,-3) node[above right] {\tikz[xscale=.4, yscale=.4]
	{
\draw[thick]
  (0,0)--(4,0)--(6,2)--(2,2)--cycle
	;
\draw
  (1,0)--(1,1)--(2,1)--(2,0) (3,1)--(2,1)--(3,0) (2,2)--(4,0)--(4,2)
  ;
	}}
	;
%quadrangle6
\draw
  (5,-3) node[above right] {\tikz[xscale=.4, yscale=.4]
	{
\draw[thick]
  (0,0)--(6,0)--(4,2)--(2,2)--cycle
	;
\draw
  (1,0)--(1,1)--(2,1)--(2,0) (3,1)--(2,1)--(3,0) (2,2)--(4,0)--(4,2)
  ;
	}}
	;
%pentagon1
\draw
  (12,-3) node[above right] {\tikz[xscale=.4, yscale=.4]
	{
\draw[thick]
  (0,0)--(4,0)--(4,1)--(2,3)--(0,1)--cycle
	;
\draw
  (0,1)--(2,1)--(2,3) (3,0)--(2,1)--(1,0)--(1,1)--(4,1) (3,1)--(4,0)
  ;
	}}
	;
%pentagon2
\draw
  (17,-3) node[above right] {\tikz[xscale=.4, yscale=.4]
	{
\draw[thick]
  (0,0)--(3,0)--(5,2)--(4,3)--(3,3)--cycle
	;
\draw
  (3,0)--(3,3) (4,3)--(3,2)--(5,2) (1,1)--(3,1) (1,0)--(2,1)--(2,0)
  ;
	}}
	;
%hexagon1
\draw
  (0,-7) node[above right] {\tikz[xscale=.4, yscale=.4]
	{
\draw[thick]
  (0,0)--(3,0)--(3,2)--(2,3)--(1,3)--(0,2)--cycle
	;
\draw
   (0,0)--(2,2)--(2,0)--(3,1)--(2,1) (0,2)--(3,2) (1,3)--(2,2)
  ;
	}}
	;
%hexagon2
\draw
  (4,-7) node[above right] {\tikz[xscale=.4, yscale=.4]
	{
\draw[thick]
  (0,0)--(2,0)--(3,1)--(3,3)--(1,3)--(0,2)--cycle
	;
\draw
  (0,0)--(2,2)--(0,2) (1,3)--(2,2)--(3,2)--(2,3) (2,2)--(2,0) (3,1)--(2,1)
  ;
	}}
	;
%hexagon3
\draw
  (8,-7) node[above right] {\tikz[xscale=.4, yscale=.4]
	{
\draw[thick]
  (0,0)--(2,0)--(4,2)--(4,3)--(2,3)--(0,1)--cycle
	;
\draw
  (0,1)--(3,1)--(3,2) (2,3)--(2,1)--(4,3) (1,0)--(1,1)--(2,0)
  ;
	}}
	;
%hexagon4
\draw
  (12,-7) node[above right] {\tikz[xscale=.4, yscale=.4]
	{
\draw[thick]
  (0,1)--(1,0)--(4,0)--(5,1)--(4,2)--(1,2)--cycle
	;
\draw
  (1,2)--(1,0)--(3,2)--(3,0)--(4,1)--(3,1) (4,2)--(4,1)--(5,1)
  ;
	}}
	;
\end{tikzpicture}
\end{center}
\caption{All $13$ convex tangrams.
\label{fig:convex}}
\end{figure}

This motivates the question for other natural classes of tangrams. When asking for all simple $n$-gons for fixed $n$, the cases of triangles and quadrangles appear trivial or simple, see Section~\ref{sec:triangles_quadrangles}. Already for hexagons one gets uncountably many incongruent tangrams: for example, one may shift the right part of the dissection on the right-hand side of Figure~\ref{fig:pieces} slightly up, this way obtaining a continuum of hexagons. Similarly, one finds uncountably many simple $n$-gons for all $n=7,\ldots,23$. The number of vertices of any tangram is at most $23$, which is the total number of vertices of all seven tans.

Here we characterize all simple pentagonal tangrams, this way answering a question that seems to have been posed by Lindgren in 1968 (cf.\ \cite{gardner1974}). Most of them are non-convex, see e.g.\ Figure~\ref{fig:pieces}. In the case of pentagons we observe a technical problem concerning the respective position of the tans: In the present paper a \emph{lattice} always means an isometric image of $\mathbb{Z}^2$. Every tan induces a unique lattice that contains all its vertices. If a tangram admits a dissection such that all tans induce the same lattice, we call it a \emph{lattice tangram}. Otherwise we call it a \emph{non-lattice tangram}. Figure~\ref{fig:pieces} illustrates both situations. Simple polygons that are lattice tangrams have been called \emph{snug tangrams} by Read \cite{gardner1974,read1985,read2004}. 

All convex tangrams from Theorem~\ref{thm:convex} are lattice tangrams. We shall obtain the following counterpart on pentagons.

\begin{theorem}\label{thm:pentagons}
There exist, up to isometry, exactly $53$ simple pentagons that are tangrams: two convex ones, $20$ non-convex lattice ones and $31$ non-convex non-lattice ones.
\end{theorem}

These pentagons are given in detail in Section~\ref{sec:pentagons}. Before that, we comment briefly the situation for triangular and quadrangular tangrams (Section~\ref{sec:triangles_quadrangles}), and we provide a topological tool that permits a systematic approach to pentagons (Section~\ref{sec:lemma}). 

The present paper is an extended version of \cite{pohl2019}.

%%%%%%%%%%%%%%%%%%%%%%%%%%%%%%%%%%%%%%%%%%%%%%%%%%%

\section{Triangular and quadrangular tangrams\label{sec:triangles_quadrangles}}

\begin{proposition}\label{prop:34}
There exist, up to isometry, only one triangular and six quadrangular tangrams. All of them are convex.
\end{proposition}

Before proving Proposition~\ref{prop:34}, we note two obvious properties of tangrams.

\begin{lemma}\label{lem:trivial}
\begin{itemize}
\item[(i)] 
The area of every tangram is $8$.
\item[(ii)]
If a tangram is a simple polygon then the sizes of its inner angles are integer multiples of $\frac{\pi}{4}$.
\end{itemize}
\end{lemma}

\begin{proof}[Proof of Proposition~\ref{prop:34}]
By Lemma~\ref{lem:trivial}, only the isosceles right triangle with legs of length $4$ can be a triangular tangram. This triangle is indeed a tangram (cf.\ \cite{wang} or Figure~\ref{fig:convex}).

All six convex quadrangular tangrams are known from \cite{wang}, see Figure~\ref{fig:convex}. It remains to show that there are no non-convex quadrangular tangrams. For that, assume that we are given such a quadrangle $Q$. Since its inner angles sum up to $2\pi$ and satisfy Lemma~\ref{lem:trivial}(ii), their sizes must be $\frac{\pi}{4}$, $\frac{\pi}{4}$, $\frac{\pi}{4}$ and $\frac{5\pi}{4}$. Let $\xi,\eta \in \mathbb{R}$ be the lengths of the sides not emanating from the non-convex vertex, see Figure~\ref{fig:quadrangle}.
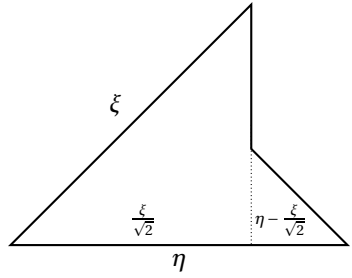
\begin{figure}
\begin{center}
\begin{tikzpicture}[xscale=.8, yscale=.8]

\draw[thick]
  (0,0)--(4,4)--(4,1.6)--(5.6,0)--cycle
	(2,2) node[above left] {$\xi$}
	(2.8,0) node[below] {$\eta$}
	(2.2,0) node[above] {\scriptsize$\frac{\xi}{\sqrt{2}}$}
	(4.5,0) node[above] {\scriptsize$\eta-\frac{\xi}{\sqrt{2}}$}
	;
\draw[densely dotted]
  (4,2)--(4,0)
  ;
	
\end{tikzpicture}
\end{center}
\caption{A potential non-convex quadrangular tangram.
\label{fig:quadrangle}}
\end{figure}
We have, w.l.o.g.,
\begin{equation}\label{eq:4estimate1}
\xi \ge \eta > \frac{\xi}{\sqrt{2}}.
\end{equation}
Lemma~\ref{lem:trivial}(i) together with \eqref{eq:4estimate1} gives
\begin{equation}\label{eq:4area}
8=\frac{1}{2}\left(\frac{\xi}{\sqrt{2}}\right)^2+\frac{1}{2}\left(\eta-\frac{\xi}{\sqrt{2}}\right)^2
\;\left\{
\begin{array}{l}
\displaystyle > \frac{1}{2}\left(\frac{\xi}{\sqrt{2}}\right)^2=\frac{1}{4}\xi^2,\\[2ex]
\displaystyle\le \frac{1}{2}\left(\frac{\xi}{\sqrt{2}}\right)^2+\frac{1}{2}\left(\xi-\frac{\xi}{\sqrt{2}}\right)^2=\left(\frac{2-\sqrt{2}}{2}\right)\xi^2.
\end{array}
\right.
\end{equation}
The two inequalities give
\begin{equation}\label{eq:4estimate2}
5.226\ldots=\sqrt{8\left(2+\sqrt{2}\right)}\le\xi < 4\sqrt{2}=5.656\ldots
\end{equation}
Since $\xi$ and $\eta$ are sums of side lengths of tans, which are integers or integer multiples of $\sqrt{2}$, we obtain
\begin{equation}\label{eq:4representation}
\xi,\eta \in \left\{k+l\sqrt{2}: k,l \in \{0,1,\ldots\}\right\}.
\end{equation}
Then \eqref{eq:4estimate2} implies
\[
\xi \in \left\{ 4+\sqrt{2},1+3\sqrt{2}\right\}.
\]
For any of the two choices of $\xi$, we see that the left-hand equation from \eqref{eq:4area} does not have a solution $\eta$ that satisfies \eqref{eq:4representation}.
\end{proof}

%%%%%%%%%%%%%%%%%%%%%%%%%%%%%%%%%%%%%%%%%%%%%%%%%%%

\section{A topological lemma\label{sec:lemma}}

Now we work in a slightly generalized setting. An isosceles right triangle with legs of length $1$ is called a \emph{basic triangle}. A polygon admitting a dissection into finitely many basic triangles is called a \emph{generalized tangram}. Of course, every tangram is a generalized tangram, since every tan can be subdivided into basic triangles (see the left-hand part of Figure~\ref{fig:notations}). Although that subdivision is not unique (namely, for the square tan as well as for the large triangular tans), the lattice associated to every basic triangle coincides with that of the original tan. So the concepts of lattice and non-lattice tangrams extend naturally to generalized tangrams.

In the remainder of this section we shall prove the following.

\begin{lemma}
\label{lem}
Let $T$ be a generalized tangram that is a simple polygon having exactly one non-convex vertex $v_1$. Then one of the following is satisfied.
\begin{itemize}
\item[(I)] All vertices of $T$ belong to the same lattice $\Lambda=\varphi\left(\mathbb{Z}^2\right)$, where $\varphi$ is a Euclidean isometry. Every side of $T$ is parallel to one of the segments $\varphi((0,0)(1,0))$, $\varphi((0,0)(0,1))$, $\varphi((0,0)(1,1))$ or $\varphi((0,0)(1,-1))$.
\item[(II)] One of the straight lines defined by a side of $T$ emanating from $v_1$ dissects $T$ into two convex generalized tangrams $T_1$ and $T_2$ with corresponding lattices $\Lambda_1$ and $\Lambda_2$ as in (I). The lattice $\Lambda_2$ is the image of $\Lambda_1$ under a rotation by $\frac{\pi}{4}$.
\end{itemize}
\end{lemma}

%%%%%%%%%%%%%%%%%%%%

\subsection{Preparations}

Suppose now that we are given a generalized tangram $T$ satisfying the assumptions of Lemma~\ref{lem}, along with a dissection into basic triangles. Two basic triangles are called equally oriented if their associated lattices differ at most by a translation. Since $T$ is a simple polygon, there are at most two different orientations. Let $P$ be the union of one class of equally oriented basic triangles, and let $Q$ be the union of the remaining basic triangles. Then $Q$ is either empty or the lattices associated to the pieces of $Q$ are obtained from those of $P$ by rotations of $\frac{\pi}{4}$. Two basic triangles of $P$ are considered equivalent if their lattices agree. Let $P_1,\ldots,P_m$ be the respective unions of all classes of equivalent basic triangles of $P$. Similarly, $Q$ splits into unions $Q_1,\ldots,Q_n$ of equivalent triangles. This way we obtain the dissection
\begin{equation}\label{eq:dissection}
T=(P_1\cup\ldots\cup P_m)\cup(Q_1\cup\ldots\cup Q_n).
\end{equation}
(Figure~\ref{fig:notations} illustrates an eight-angled tangram with corresponding dissections into tans as well as into basic triangles, on the left, and the resulting dissection introduced above, on the right.)
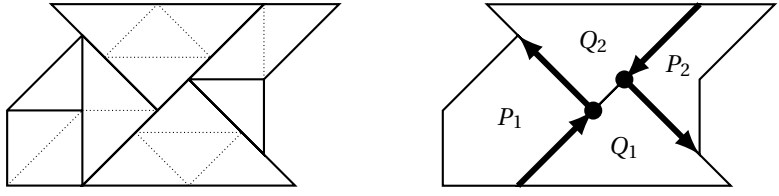
\begin{figure}
\begin{center}
\begin{tikzpicture}
\draw[densely dotted] 
	(0,0)--(1,1) (1,1)--(2,1) (3.414,1.414)--(3.414,2.414) 
	(1.707,.707)--(2.414,0)--(3.121,.707)--cycle
	(1.293,1.707)--(2.707,1.707)--(2,2.414)--cycle
  ;
\draw[thick]
  (0,0)--(1,0)--(2,1)--(1,2)--(0,1)--cycle
	(1,0)--(3.828,0)--(2.414,1.414)--cycle
	(2,1)--(3.414,2.414)--(.586,2.414)--cycle
	(2.414,1.414)--(3.414,2.414)--(4.414,2.414)--(3.414,1.414)--(3.414,.414)--cycle
	
	(0,1)--(1,1) (1,0)--(1,2)
	(2.414,1.414)--(3.414,1.414)
	;
\draw
  (8,1.207) node {
	\tikz{
\draw[thick]
  (0,0)--(1,0)--(2,1)--(1,2)--(0,1)--cycle
	(1,0)--(3.828,0)--(2.414,1.414)--cycle
	(2,1)--(3.414,2.414)--(.586,2.414)--cycle
	(2.414,1.414)--(3.414,2.414)--(4.414,2.414)--(3.414,1.414)--(3.414,.414)--cycle
  ;	
\fill
  (2,1) circle (1.2mm)
	(2.414,1.414) circle (1.2mm)
	;
\draw[line width=.8mm,-latex]
  (1,0)--(2,1)
  ;
\draw[line width=.8mm,-latex]
	(3.414,2.414)--(2.414,1.414)
  ;
\draw[line width=.8mm,-latex]
  (2,1)--(1,2)
	;
\draw[line width=.8mm,-latex]
	(2.414,1.414)--(3.414,.414)
  ;
\draw
  (.9,.9) node {$P_1$}	
  (2.414,.5) node {$Q_1$}	
  (2,1.914) node {$Q_2$}	
  (3.13,1.6) node {$P_2$}	
	}}
	;
\end{tikzpicture}
\end{center}
\caption{
A tangram with a subdivision into basic triangles (left) and concepts for the proof of Lemma~\ref{lem} (right).
\label{fig:notations}}
\end{figure}
 
We cannot assume the polygons $P_i$ and $Q_j$ to be simple or even connected. In that setting we call an element $v$ of the boundary $\BD(R)$ of a polygon $R$ a \emph{vertex of $R$} if there is no circular disc $D$ centered at $v$ such that $D \cap R$ is a half-disc of $D$.
When speaking of a \emph{side of $R$}, we mean a maximal line segment in $\BD(R)$ that does not contain one of the vertices 
of $R$ in its relative interior.
 
\begin{fact}\label{fact:common_vertices}
Let $1 \le i < i' \le m$ and $1 \le j < j' \le n$. Then 
\begin{itemize}
\item[(i)] 
$P_{i}$ and $P_{i'}$ (and, similarly, $Q_{j}$ and $Q_{j'}$) do not have a common vertex,
\item[(ii)] 
$P_i$ and $Q_j$ have at most one common vertex.
\end{itemize}
\end{fact}

\begin{proof}
(i): A common vertex of $P_i$ and $P_{i'}$ would belong to both the lattices associated to $P_i$ and $P_{i'}$, respectively. But this intersection is empty, because these lattices do not agree and are translates of each other.

(ii): The lattices associated to $P_i$ and $Q_j$ are images of each other under a rotation by $\frac{\pi}{4}$. Thus they have at most one point in common, the only possible joint vertex of $P_i$ and $Q_j$. 
\end{proof}

The \emph{skeleton} $S$ of the dissection in \eqref{eq:dissection} is the union of the boundaries of its pieces, i.e.,
\[
S=\BD(P_1) \cup \ldots\cup \BD(P_m) \cup \BD(Q_1) \cup\ldots\cup \BD(Q_n).
\]
We shall deal with arcs contained in $S$. We say that an arc starts (or ends) in $P_i$ if a segment of positive length of the beginning (or the end) of that arc is contained in $\BD(P_i)$.

\begin{fact}\label{fact:arc_in_Pi}
Let $\Theta \subseteq \BD(T) \cap P$ be an arc that does not contain the non-convex vertex $v_1$ in its relative interior. If $\Theta$ starts in $P_i$, $1 \le i \le n$, then $\Theta \subseteq P_i$. (The analogue applies to arcs in $\BD(T) \cap Q$.)
\end{fact}

\begin{proof}
Assume that $\Theta \not\subseteq P_i$. Then there is a point $x_0 \in \Theta$ where $\Theta$ switches from $P_i$ into some $P_{i'}$ with $i' \ne i$. Since $\Theta$ does not pass through the only non-convex vertex $v_1$ of $T$, the point $x_0$ must be a common vertex of $P_i$ and $P_{i'}$. This contradicts Fact~\ref{fact:common_vertices}(i).
\end{proof}

Let $v_1,\ldots,v_k$ be the vertices of $T$ appearing along its boundary in counter-clockwise direction. These are illustrated in the left-hand part of Figure~\ref{fig:arc}.
\begin{figure}
\begin{center}
\begin{tikzpicture}
\draw
  (2,1.9) node {\tikz
	{
\draw[thick]
	(2.5,4)--(2.4,4) (2.3,4)--(1,4)--(0,3)--(0,2)--(2,2)--(1,1)--(1,0)--(2.3,0) (2.4,0)--(2.5,0) 
	(1.4,2) node[below] {$v_1$}
	(1,1) node[left] {$v_2$}
	(1,0) node[left] {$v_3$}
	(0,2) node[left] {$v_k$}
	(0,3) node[left] {$v_{k-1}$}
	(1.9,1.5) node {$P_1$}
	;
	}}
	;
\fill
  (7.5,1) circle (1.2mm)
  (6,1) circle (1.2mm)
  (5,2) circle (1.2mm)
  (7,2) circle (1.2mm)
  (7,3) circle (1.2mm)
	;
\draw[line width=.8mm,-latex]
  (9,2.5)--(7.5,1)
  ;
\draw[line width=.8mm,-latex]
  (7.5,1)--(6,1)
  ;
\draw[line width=.8mm,-latex]
  (6,1)--(5,2)
  ;
\draw[line width=.8mm,-latex]
  (5,2)--(7,2)
  ;
\draw[line width=.8mm,-latex]
  (7,2)--(7,3)
  ;
\draw[line width=.8mm,-latex]
  (7,3)--(5.5,3)
  ;
\draw[thick]
	%(7.3,4)--(7.4,4) 
	(7.5,4)--(7.6,4) (7.7,4)--(9,4)--(9,1)--(8,0)--(7.7,0) (7.6,0)--(7.5,0) 
	%(7.4,0)--(7.3,0)
	(7.5,1)--(7.22,.72) (7.15,.65)--(7.08,.58)
	(6,1)--(5.6,1) (5.5,1)--(5.4,1)
	(5,2)--(4.72,2.28) (4.65,2.35)--(4.58,2.42)
	(7,2)--(7.4,2) (7.5,2)--(7.6,2)
	(7,3)--(7,3.4) (7,3.5)--(7,3.6)
  (9,2.5) node[right] {$a_0$}
  (7.75,.75) node {$a_1$}
  (6,.69) node {$a_2$}
  (4.75,1.75) node {$a_3$}
  (7,1.69) node {$a_4$}
  (7.37,3) node {$a_5$}
  (8.6,1.6) node {$P_{i_1}$}	
  (6.8,.75) node {$P_{i_2}$}	
  (5.25,1.4) node {$P_{i_3}$}	
  (5.75,2.25) node {$P_{i_4}$}	
  (6.7,2.4) node {$P_{i_5}$}	
  (6.25,2.75) node {$P_{i_6}$}	
  (8.1,2.1) node {$Q_{j_1}$}	
  (6.8,1.25) node {$Q_{j_2}$}	
  (6,1.4) node {$Q_{j_3}$}	
  (6.25,1.75) node {$Q_{j_4}$}	
  (7.35,2.4) node {$Q_{j_5}$}	
  (6.25,3.25) node {$Q_{j_6}$}	
	;
\end{tikzpicture}
\end{center}
\caption{The vertices of $T$ and the arc $\Delta$.
\label{fig:arc}}
\end{figure}
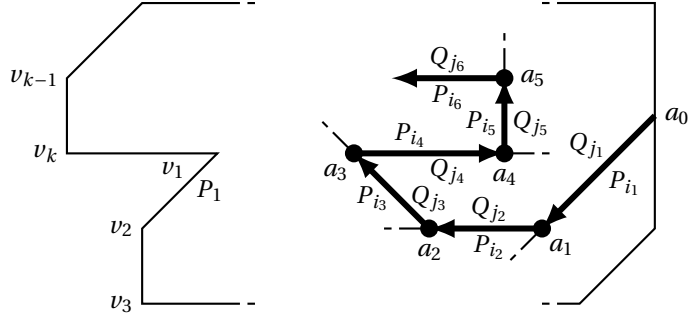

\begin{fact}\label{fact:start(II)}
From now on we can assume that
\begin{itemize}
\item[(i)]
the arc $v_1v_2\ldots v_k$ starts in $P_1$ and $v_1$ is a vertex of $P_1$,
\item[(ii)]
$\BD(T) \not\subseteq P$ and $\BD(T) \not\subseteq Q$.
\end{itemize}
\end{fact}

\begin{proof}
(i): There exists one of the pieces $P_i$, $1 \le i \le m$, or $Q_j$, $1 \le j \le n$, such that $v_1$ is a vertex of that piece and one of the segments $v_1v_2$ or $v_1v_k$ starts in that piece. By possibly renaming the pieces in \eqref{eq:dissection} we can assume that that piece is $P_1$. By possibly reversing the order of vertices of $T$ we can assume that $v_1v_2$ starts in $P_1$, w.l.o.g.

(ii): If $\BD(T) \subseteq P$, then application of Fact~\ref{fact:arc_in_Pi} to the arc $\Theta=v_1v_2\ldots v_k$, which starts in $P_1$ by (i), yields $\Theta \subseteq P_1$. Consequently, all vertices $v_1,\ldots,v_k$ of $T$ are vertices of $P_1$, which belong to the same lattice according to the definition of $P_1$. This way we reach claim (I) of Lemma~\ref{lem}, and the proof is complete if $\BD(T) \subseteq P$. Thus the case $\BD(T) \subseteq P$ does not require further consideration.

The situation $\BD(T) \subseteq Q$ does not appear, since $v_1v_2$ starts in $P_1$ by (i).
\end{proof}

%%%%%%%%%%%%%%%%%%%%%%%%%

\subsection{The arc $\Delta$}

We call an interior point $x$ of $T$ a \emph{T-vertex} of the dissection \eqref{eq:dissection} if $x$ is a joint vertex of two pieces $P_i$ and $Q_j$ and belongs to the relative interior of a side of some $P_{i'}$, $i'\ne i$, or $Q_{j'}$, $j'\ne j$. That is, a side of $P_{i'}$ or $Q_{j'}$ passes through $x$ and at least one joint segment of $P_i$ and $Q_j$ emanates from $x$ and cannot be continued over $x$ within the skeleton $S$. Then we say that $P_{i}$ and $Q_{j}$ are \emph{below} the T-vertex $x$ and $P_{i'}$ (or $Q_{j'}$, respectively) is \emph{above} $x$. (Note that the name T-vertex is motivated by the shape of the capital letter T, but not by the name $T$ of our polygon.)

A directed line segment $xy \subseteq S$ is called a \emph{primal segment} if there exist $i$ and $j$ such that $xy \subseteq P_i \cap Q_j$, $x$ is a joint vertex of $P_i$ and $Q_j$, $y$ is a vertex of one of $P_i$ or $Q_j$, and $xy$ does not contain further vertices of $P_i$ or $Q_j$. Figure~\ref{fig:notations} illustrates all T-vertices as small circles and all primal segments as arrows.

\begin{fact}\label{fact:primal}
Let $xy \subseteq P_i \cap Q_j$ be a primal segment. Then either $y=v_1$ and $v_1$ is in the relative interior of some side of one of $P_i$ or $Q_j$ or $y$ is a T-vertex having one of $P_i$ or $Q_j$ above.
\end{fact}

\begin{proof}
Suppose that $y$ is a vertex of $P_i$, w.l.o.g. By Fact~\ref{fact:common_vertices}(ii), $y$ is no vertex of $Q_j$. So $y$ is in the relative interior of a side of $Q_j$.

If $y$ is in the interior of $T$, Fact~\ref{fact:common_vertices} implies that $y$ is a vertex of some $Q_{j'}$, $j'\ne j$. Then $y$ is a T-vertex with $P_i$ and $Q_{j'}$ below and $Q_j$ above.

Now let $y \in \BD(T)$. If $y \ne v_1$ then $y$ had to be be a vertex of both $P_i$ and $Q_j$, since $v_1$ is the only non-convex vertex of $T$. This contradiction shows that $y=v_1$.
\end{proof}

Next we define a particular arc $\Delta= a_0a_1\ldots a_l$ in the skeleton $S$ as a union of primal segments (see Figure~\ref{fig:arc}). We know from Fact~\ref{fact:start(II)} that, when starting in $v_1$ and following $\BD(T)$ counter-clockwise, we meet a first point $a_0 \in \BD(T) \setminus \{v_1\}$ where we switch from $P$ to $Q$. There exist (by Fact~\ref{fact:common_vertices}(i) unique) $1 \le i_1 \le m$ and $1 \le j_1 \le n$ such that $a_0$ is a vertex of both $P_{i_1}$ and $Q_{j_1}$. There is (at least) one primal segment $a_0a_1 \subseteq P_{i_1} \cap Q_{j_1}$ such that $P_{i_1}$ and $Q_{j_1}$ are left and right beside $a_0a_1$, respectively. This segment is the beginning of $\Delta$.

Now we continue the definition of $\Delta$ recursively. Suppose that $a_{r-1}a_r \subseteq P_{i_r} \cap Q_{j_r}$ is the last chosen segment in $\Delta$. By Fact~\ref{fact:primal}, we can continue as follows.
\begin{itemize}
\item
If $a_r=v_1$, we put $l=r$ and $\Delta$ is complete.
\item
If $a_r$ is a T-vertex with $P_{i_r}$ above, there is a (by Fact~\ref{fact:common_vertices}(i) unique) $1 \le i_{r+1} \le m$ such that $P_{i_{r+1}}$ is below $a_r$. We put $j_{r+1}=j_r$ and continue $\Delta$ with a next primal segment $a_r a_{r+1} \subseteq P_{i_{r+1}} \cap Q_{j_{r+1}}$ having $P_{i_{r+1}}$ on the left and $Q_{j_{r+1}}$ on the right.
\item
If $a_r$ is a T-vertex with $Q_{j_r}$ above, we proceed analogously
and obtain a next primal segment $a_r a_{r+1} \subseteq P_{i_{r+1}} \cap Q_{j_{r+1}}$ having $P_{i_{r+1}}$ on the left and $Q_{j_{r+1}}$ on the right, where $i_{r+1}=i_r$.
\end{itemize}

Note that this definition terminates after finitely many steps, since it produces no loops. Indeed, suppose that $x_0$ would be a first point of self-intersection. Then $x_0$ cannot be in the relative interior of a primal segment, because these do not contain vertices. Hence $x_0$ is a T-vertex with $x_0=a_r=a_{r'}$, $r<r'$, and the primal segments $a_{r-1}a_r$ and $a_{r'-1}a_{r'}$ arrive at $x_0$ from opposite directions. This is impossible, because both segments have $P$ on the left and $Q$ on the right.

In the following we use the notation $\widearc{xy}$, $x,y \in \BD(T)$, for the counter-clockwise arc in $\BD(T)$ that begins in $x$ and ends in $y$.

\begin{fact}\label{fact:switch}
When following $\BD(T)$ counter-clockwise, $a_0$ is the only point where one switches from $P$ to $Q$. 
Accordingly, there exists exactly one point $b_0 \in \widearc{a_0 v_1} \setminus \{a_0\}$ where one switches back from $Q$ to $P$.
\end{fact}

\begin{proof}
First note that we can recover $\Delta$ uniquely by the following backwards procedure: Applying Fact~\ref{fact:primal} to the last segment $a_{l-1}a_l=a_{l-1}v_1$ of $\Delta$, we see that $v_1$ is in the relative interior of a side of one of $P_{i_l}$ or $Q_{j_l}$. Since there is at most one of the $m+n$ tiles from \eqref{eq:dissection} having $v_1$ in the relative interior of one of its sides, this determines that very tile uniquely. Moreover, since $P$ is on the left and $Q$ is on the right of $\Delta$, this recovers at least a little segment at the end of $\Delta$. Now we find $a_{l-1}$ by following that direction until we end at a T-vertex (or at $\BD(T)$, where we meet $a_0$). The polygon above that T-vertex $a_{l-1}$ is $P_{i_{l-1}}$ or $Q_{j_{l-1}}$. Again by using the lateral position of $P$ and $Q$ relative to $\Delta$ we determine the direction of the segment $a_{l-2} a_{l-1}$. Following that direction until the next T-vertex (or until $\BD(T)$, where we find $a_0$) we find $a_{l-2}$. Continuing this way we reproduce $\Delta$ uniquely.

Now, for proving the fact, let $a_0' \in \BD(T)$ be a point where one switches from $P$ to $Q$. Starting from $a_0'$, we define an arc $\Delta' \subseteq S$ from $a_0'$ to $v_1$ by the same rules as we did with $\Delta$ from $a_0$. We can reproduce $\Delta'$ beginning at its end $v_1$, as we did with $\Delta$. But this reproduction gives the same arc $\Delta'=\Delta$, so that $a_0'=a_0$.
\end{proof}

\begin{fact}\label{fact:piece_at_v1}
Let $R_0$ be the piece among $P_1,\ldots,P_m,Q_1,\ldots,Q_n$ such that the segment $v_kv_1$ ends in $R_0$. Then $v_1$ is in the relative interior of a side of $R_0$.
\end{fact}

\begin{proof}
Assume that our claim fails. Then $v_1$ is a vertex of $R_0$.

\emph{Case 1: $R_0=P_{i^*}$ for some $1 \le i^* \le m$. } Fact~\ref{fact:common_vertices}(i) yields $i^*=1$, since $v_1$ is a vertex of both $P_1$ and $P_{i^*}$. Applying Fact~\ref{fact:arc_in_Pi} to the arcs $\widearc{v_1a_0},\widearc{b_0v_1} \subseteq P$, we see that $a_0$ and $b_0$ are vertices of $P_1$. Similarly, applying Fact~\ref{fact:arc_in_Pi} to $\widearc{a_0b_0} \subseteq Q$, we see that both $a_0$ and $b_0$ are vertices of some $Q_j$. This contradicts Fact~\ref{fact:common_vertices}(ii).

\emph{Case 2: $R_0=Q_{j^*}$ for some $1 \le j^* \le n$. } Now $b_0=v_1$, since we switch from $Q_{j*}$ to $P_1$ in $v_1$. Applications of Fact~\ref{fact:arc_in_Pi} to $\widearc{v_1 a_0} \subseteq P$ and $\widearc{a_0 v_1} \subseteq Q$ shows that $v_1$ and $a_0$ are joint vertices of both $P_1$ and $Q_{j^*}$, again contradicting Fact~\ref{fact:common_vertices}(ii).
\end{proof}

Now we learn more on the end of $\Delta=a_0a_1\ldots a_l$.

\begin{fact}\label{fact:end_Delta}
The points $a_{l-1}$, $a_l=v_1$ and $v_k$ are collinear. The vertex $v_1$ is in the relative interior of a side of $Q_{j_l}$.
In particular, $b_0=v_1$.
\end{fact}

\begin{proof}
There is at most one piece among $P_1,\ldots,P_m,Q_1,\ldots,Q_n$ having $v_1$ in the relative interior of a side. By Fact~\ref{fact:piece_at_v1}, such a piece $R_0$ exists and $v_1v_k$ represents the direction of that side. On the other hand,
application of Fact~\ref{fact:primal} to the last primal segment $a_{l-1}a_l=a_{l-1}v_1$ of $\Delta$ shows that one of $P_{i_l}$ or $Q_{j_l}$ has $v_1$ in the relative interior of a side. Consequently, $R_0 \in \left\{P_{i_l},Q_{j_l}\right\}$ and $a_{l-1}$, $a_l=v_1$ and $v_k$ are collinear. Since $P_{i_l}$ is on the left of $a_{l-1}a_l$ and since $v_kv_1$ ends in $R_0$, we obtain $R_0=Q_{j_l}$. Finally, since $v_kv_1$ ends in $R_0=Q_{j_l}$, Fact~\ref{fact:switch} yields $b_0=v_1$.
\end{proof}

Figure~\ref{fig:almost_done} gives a scheme of we have reached so far.
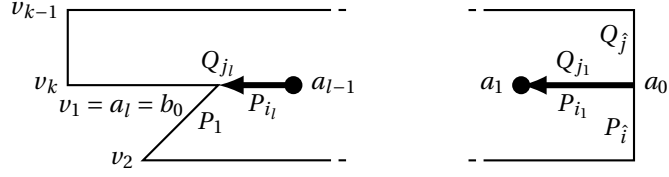
\begin{figure}
\begin{center}
\begin{tikzpicture}
\draw
  (3,2) node {\tikz
	{
\draw[thick]
	(3.7,3)--(3.6,3) (3.5,3)--(0,3)--(0,2)--(2,2)--(1,1)--(3.5,1) (3.6,1)--(3.7,1) 
	(.7,2) node[below] {$v_1=a_l=b_0$}
	(1,1) node[left] {$v_2$}
	(0,2) node[left] {$v_k$}
	(0,3) node[left] {$v_{k-1}$}
	(1.9,1.5) node {$P_1$}
	(2.6,2) node[below] {$P_{i_l}$}
	(2,2) node[above] {$Q_{j_l}$}
	(3.1,2) node[right] {$a_{l-1}$}
	;
\draw[line width=.8mm,-latex]
  (3,2)--(2,2)
  ;
\fill
  (3,2) circle (1.2mm)
  ;
	}}
	;
\fill
  (7.5,2) circle (1.2mm)
	;
\draw[line width=.8mm,-latex]
  (9,2)--(7.5,2)
  ;
\draw[thick]
	(6.8,3)--(6.9,3) (7,3)--(9,3)--(9,1)--(7,1) (6.9,1)--(6.8,1)
  (9,2) node[right] {$a_0$}
  (7.4,2) node[left] {$a_1$}
  (8.2,2) node[below] {$P_{i_1}$}	
  (8.2,2) node[above] {$Q_{j_1}$}	
  (9.05,1.4) node[left] {$P_{\hat{i}}$}	
  (9.05,2.6) node[left] {$Q_{\hat{j}}$}	
	;
\end{tikzpicture}
\end{center}
\caption{The situation after Fact~\ref{fact:end_Delta}.
\label{fig:almost_done}}
\end{figure}

%%%%%%%%%%%%%%%%%%%%%%%%%

\subsection{Conclusion}

The counter-clockwise arcs $\widearc{v_1a_0},\widearc{a_0v_1} \subseteq \BD(T)$ are contained in $P$ and $Q$, respectively. Thus there are $1 \le \hat{i} \le m$, $1 \le \hat{j} \le n$ such that $\widearc{v_1a_0}$ ends in $P_{\hat{i}}$ and $\widearc{a_0v_1}$ begins in $Q_{\hat{j}}$. Since $a_0$ is a vertex of $P_{\hat{i}}$, $P_{i_1}$, $Q_{\hat{j}}$ and $Q_{j_1}$, Fact~\ref{fact:common_vertices}(i) yields $P_{\hat{i}}=P_{i_1}$ and $Q_{\hat{j}}=Q_{j_1}$. Similarly, $P_{i_l}=P_1$, because of the joint vertex $a_l=v_1$.
Moreover, we apply Fact~\ref{fact:arc_in_Pi} to $\widearc{v_1a_0} \subseteq P$ and $\widearc{a_0v_1}\subseteq Q$ and obtain
\begin{equation}\label{eq:final_cut}
\widearc{v_1a_0} \subseteq P_{i_1}=P_{\hat{i}}=P_1=P_{i_l},\qquad
\widearc{a_0v_1} \subseteq Q_{j_1}=Q_{\hat{j}}=Q_{j_l}.
\end{equation}
In particular, both $a_0$ and $a_{l-1}$ are joint vertices of $P_{i_1}=P_{i_l}$ and $Q_{j_1}=Q_{j_l}$. Then Fact~\ref{fact:common_vertices}(ii) implies $a_0=a_{l-1}$, and in turn $l=1$. That is, the arc $\Delta$ consists of the single segment $\Delta=a_0a_1=a_0v_1$.

We see that $\Delta$ dissects $T$ along the straight line spanned by $v_1v_k$ into two convex generalized tangrams. By \eqref{eq:final_cut}, all vertices of one of them are vertices of $P_{i_1}$ and all vertices of the other one are vertices of $Q_{j_1}$. This completes the proof of Lemma~\ref{lem}.

%%%%%%%%%%%%%%%%%%%%%%%%%%%%%%%%%%%%%%%%%%%%%%%%%%%

\section{The complete list of pentagonal tangrams\label{sec:pentagons}}

In this section we prove Theorem~\ref{thm:pentagons}. We distinguish the cases of convex and non-convex tangrams, the latter case being split into the two alternatives according to Lemma~\ref{lem}.

%%%%%%%%%%%%%%

\subsection{Convex pentagons}

\begin{proposition}
There exist, up to isometry, exactly two convex pentagonal tangrams.
\end{proposition}

This follows from Theorem~\ref{thm:convex}, see Figure~\ref{fig:convex} for an illustration.

%%%%%%%%%%%%%%%%

\subsection{Non-convex lattice pentagons}

Let a simple non-convex pentagon $T$ be a tangram. Since the sizes of its inner angles are integer multiples of $\frac{\pi}{4}$ (see Lemma~\ref{lem:trivial}(ii)) and add up to $3\pi$, $T$ has only one non-convex vertex. Thus $T$ satisfies one of the situations (I) or (II) from Lemma~\ref{lem}. In the present subsection we assume that (I) applies.  

The following discussion seems to go back to Read (published in \cite{gardner1974}) and, independently, to Heinert \cite{heinert1996}. We reproduce it here, since the publication \cite{gardner1974} misses six of the $20$ solutions (see \cite[p.\ 125]{gardner1974b}) and since Heinert's manuscript \cite{heinert1996} is almost inaccessible.
Note that Read and Heinert assumed their tangrams a priorily to be lattice tangrams. We use only the situation of Lemma~\ref{lem}(I), i.e., that the boundary of $T$ is arranged along a lattice.

The inner angles of $T$ have sizes $\alpha_i=k_i\frac{\pi}{4}$ with $k_1 \in \{5,6,7\}$ (non-convex vertex) and $k_i\in \{1,2,3\}$, $2 \le i \le 5$ (convex vertices). Their sum is $\alpha_1+\ldots+\alpha_5=3\pi$, i.e., $k_1+\ldots+k_5=12$. Hence the multi-set $\{k_1,\ldots,k_5\}$ is one of $\{5,3,2,1,1\}$, $\{5,2,2,2,1\}$, $\{6,3,1,1,1\}$, $\{6,2,2,1,1\}$ or $\{7,2,1,1,1\}$. 

For fixed $\{k_1,\ldots,k_5\}$, the permutations $k_{\varrho(1)}\ldots k_{\varrho(5)}$ give the successive order of the angles along the boundary of $T$. Cyclic shifts and reversions can be considered equal, since they correspond to isometric images of $T$. Thus it remains to consider the following $16$ orders of inner angles: $53211$, $53121$, $53112$, $52311$, $52131$, $51321$, $52221$, $52212$, $63111$, $61311$, $62211$, $62121$, $62112$, $61221$, $72111$ and $71211$.

Next, for each of the $16$ possible orders $k_{\varrho(1)}\ldots k_{\varrho(5)}$, one generates all corresponding pentagons that satisfy condition (I) from Lemma~\ref{lem} and whose area is $8$ according to Lemma~\ref{lem:trivial}(i). This can be done in the lattice $\mathbb{Z}^2$. One fixes a first vertex and the direction of the first side of $T$: either $(1,0)$ or $(1,1)$, both have to be considered separately. Then the order $k_{\varrho(1)}\ldots k_{\varrho(5)}$ determines $T$ up to three integer parameters, since three side lengths fix $T$ up to isometry and the length of the $i$th side is either $s_i$ or $s_i \sqrt{2}$ with $s_i \in\{1,2,\ldots\}$. Now one can easily find the choices of side lengths that give polygons $T$ of area $8$. 

(Alternatively, one could find these non-convex pentagons $T$ of area $8$ as follows: One generates all convex generalized lattice tangrams of areas less than $8$, as it has been done in \cite{wang} for those of area $8$. Then one examines for all pairs of total area $8$ if and how they can be put together to form a non-convex lattice pentagon.)

Finally, one checks for all these resulting pentagons $T$ of area $8$ if they admit a dissection into the seven tans. This can be done by hand as well, because it turns out that such a candidate $T$ cannot be dissected only if the two large tans cannot be packed simultaneously inside $T$. 

This search results in the $20$ tangrams depicted in Figure~\ref{fig:lattice}.

\begin{figure}
\begin{center}
\begin{tikzpicture}[xscale=.4, yscale=.4]
%53121
\draw
  (0,0) node[above right] {\tikz[xscale=.4, yscale=.4]
	{
\draw[thick]
  (0,2)--(2,0)--(6,4)--(4,4)--(2,2)--cycle
	(2,0) node[below] {$53121$}
	;
\draw
  (4,4)--(4,2)--(2,2)--(2,1)--(3,2) (1,2)--(1,1)--(3,1)
  ;
	}}
	;
%53112
\draw
  (4,0) node[above right] {\tikz[xscale=.4, yscale=.4]
	{
\draw[thick]
  (0,0)--(8,0)--(7,1)--(3,1)--(2,2)--cycle
	(4,0) node[below] {$53112$}
	;
\draw
  (2,0)--(2,2) (3,1)--(4,0)--(5,1) (5,0)--(6,1)--(6,0) (7,0)--(7,1)
  ;
	}}
	;
%52311
\draw
  (12,1) node[above right] {\tikz[xscale=.4, yscale=.4]
	{
\draw[thick]
  (0,0)--(5,0)--(4,1)--(4,3)--(3,3)--cycle
	(2.5,0) node[below] {$52311$}
	;
\draw
  (2,0)--(2,2)--(4,0) (4,1)--(3,1)--(3,3) (4,3)--(3,2)--(4,2)
  ;
	}}
	;
%52131.a
\draw
  (18,1) node[above right] {\tikz[xscale=.4, yscale=.4]
	{
\draw[thick]
  (0,0)--(4,0)--(4,2)--(5,3)--(3,3)--cycle
	(2,0) node[below] {$52131$.a}
	;
\draw
  (2,0)--(2,2)--(4,0) (3,3)--(3,1) (4,3)--(4,2)--(3,2)--(4,1)
  ;
	}}
	;
%52131.b
\draw
  (22,0) node[above right] {\tikz[xscale=.4, yscale=.4]
	{
\draw[thick]
  (0,0)--(7,0)--(8,1)--(3,1)--(2,2)--cycle
	(3.5,0) node[below] {$52131$.b}
	;
\draw
  (2,0)--(2,2) (3,1)--(4,0)--(5,1) (5,0)--(6,1)--(6,0) (7,0)--(7,1)
  ;
	}}
	;
%51321.a
\draw
  (0,-5) node[above right] {\tikz[xscale=.4, yscale=.4]
	{
\draw[thick]
  (0,0)--(3,0)--(4,1)--(3,1)--(0,4)--cycle
	(1.5,0) node[below] {$51321$.a}
	;
\draw
  (0,2)--(2,2)--(0,0) (1,0)--(1,1)--(3,1)--(2,0)--(2,1)
  ;
	}}
	;
%51321.b
\draw
  (5,-5) node[above right] {\tikz[xscale=.4, yscale=.4]
	{
\draw[thick]
  (0,0)--(2,0)--(4,2)--(2,2)--(0,4)--cycle
	(1.5,0) node[below] {$51321$.b}
	;
\draw
  (0,0)--(2,2)--(2,0) (0,1)--(1,1)--(1,3) (0,2)--(1,2)--(0,3)
  ;
	}}
	;
%51321.c
\draw
  (10,-5) node[above right] {\tikz[xscale=.4, yscale=.4]
	{
\draw[thick]
  (0,0)--(1,0)--(4,3)--(1,3)--(0,4)--cycle
	(1.5,0) node[below] {$51321$.c}
	;
\draw
  (0,0)--(1,1)--(0,2)--(2,2)--(1,3) (1,2)--(1,1)--(2,1)--(2,3)
  ;
	}}
	;
%51321.d
\draw
  (15,-5) node[above right] {\tikz[xscale=.4, yscale=.4]
	{
\draw[thick]
  (0,2)--(2,0)--(5,3)--(5,5)--(2,2)--cycle
	(2,0) node[below] {$51321$.d}
	;
\draw
  (1,2)--(1,1)--(3,1)--(3,3)--(5,3) (2,2)--(2,1)--(3,2)
  ;
	}}
	;
%52221
\draw
  (21,-5) node[above right] {\tikz[xscale=.4, yscale=.4]
	{
\draw[thick]
  (0,0)--(4,0)--(2,2)--(2,3)--(0,3)--cycle
	(2,0) node[below] {$52221$}
	;
\draw
  (2,0)--(2,2)--(1,3)--(1,1) (0,2)--(1,2)--(0,1) (0,0)--(2,2)
  ;
	}}
	;
%63111
\draw
  (0,-9) node[above right] {\tikz[xscale=.4, yscale=.4]
	{
\draw[thick]
  (0,0)--(7,0)--(6,1)--(3,1)--(3,3)--cycle
	(3.5,0) node[below] {$63111$}
	;
\draw
  (2,0)--(2,2)--(4,0) (4,1)--(5,0)--(5,1) (6,0)--(6,1)
  ;
	}}
	;
%61311.a
\draw
  (7,-10) node[above right] {\tikz[xscale=.4, yscale=.4]
	{
\draw[thick]
  (0,0)--(3,0)--(5,2)--(4,2)--(4,4)--cycle
	(1.5,0) node[below] {$61311$.a}
	;
\draw
  (2,0)--(2,2)--(4,2)--(3,1)--(3,2) (4,1)--(2,1)--(3,0)
  ;
	}}
	;
%61311.b
\draw
  (12,-9) node[above right] {\tikz[xscale=.4, yscale=.4]
	{
\draw[thick]
  (0,0)--(6,0)--(7,1)--(3,1)--(3,3)--cycle
	(3,0) node[below] {$61311$.b}
	;
\draw
  (2,0)--(2,2)--(4,0)--(4,1) (5,1)--(5,0)--(6,1)
  ;
	}}
	;
%61311.c
\draw
  (19,-10) node[above right] {\tikz[xscale=.4, yscale=.4]
	{
\draw[thick]
  (0,0)--(5,0)--(3,2)--(5,4)--(4,4)--cycle
	(2.5,0) node[below] {$61311$.c}
	;
\draw
  (2,0)--(2,2)--(3,1)--(2,1) (3,0)--(3,3) (4,3)--(4,4)
  ;
	}}
	;
%62211
\draw
  (25,-10) node[above right] {\tikz[xscale=.4, yscale=.4]
	{
\draw[thick]
  (0,0)--(4,0)--(3,1)--(5,3)--(4,4)--cycle
	(2,0) node[below] {$62211$}
	;
\draw
  (2,0)--(2,2)--(4,2)--(4,4) (3,0)--(2,1)--(3,1)--(3,2)
  ;
	}}
	;
%62121
\draw
  (0,-14) node[above right] {\tikz[xscale=.4, yscale=.4]
	{
\draw[thick]
  (0,0)--(3,0)--(3,2)--(6,2)--(4,4)--cycle
	(1.5,0) node[below] {$62121$}
	;
\draw
  (2,0)--(2,1)--(3,0) (1,0)--(1,1)--(2,1)--(3,2)--(2,2) (4,2)--(4,4)
  ;
	}}
	;
%61221
\draw
  (7,-14) node[above right] {\tikz[xscale=.4, yscale=.4]
	{
\draw[thick]
  (0,0)--(4,0)--(4,3)--(2,1)--(0,3)--cycle
	(2,0) node[below] {$61221$}
	;
\draw
  (2,1)--(1,0)--(0,1)--(4,1) (2,0)--(3,1)--(3,0)
  ;
	}}
	;
%72111.a
\draw
  (12,-14) node[above right] {\tikz[xscale=.4, yscale=.4]
	{
\draw[thick]
  (0,0)--(5,0)--(3,2)--(4,2)--(4,4)--cycle
	(2.5,0) node[below] {$72111$.a}
	;
\draw
  (2,1)--(3,1)--(2,0)--(2,2)--(3,2)--(3,1)--(4,0)--(4,1)
  ;
	}}
	;
%72111.b
\draw
  (18,-14) node[above right] {\tikz[xscale=.4, yscale=.4]
	{
\draw[thick]
  (0,0)--(6,0)--(4,2)--(3,1)--(3,3)--cycle
	(3,0) node[below] {$72111$.b}
	;
\draw
  (1,0)--(2,1)--(2,0) (1,1)--(3,1)--(4,0)--(4,2)  (3,0)--(3,3)
  ;
	}}
	;
%71211
\draw
  (25,-14) node[above right] {\tikz[xscale=.4, yscale=.4]
	{
\draw[thick]
  (0,0)--(5,0)--(4,1)--(5,1)--(3,3)--cycle
	(2.5,0) node[below] {$71211$}
	;
\draw
  (3,0)--(3,3) (4,0)--(2,2)--(2,0) (2,1)--(4,1)
  ;
	}}
	;
\end{tikzpicture}
\end{center}
\caption{All $20$ non-convex pentagonal lattice tangrams \cite{heinert1996}.
\label{fig:lattice}}
\end{figure}
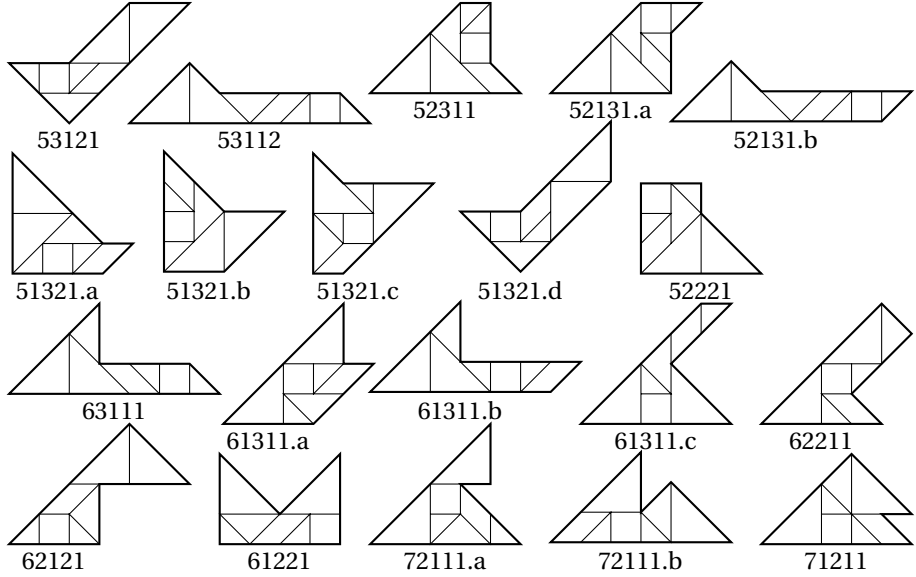

\begin{proposition}[\cite{heinert1996}]
There exist, up to isometry, exactly $20$ non-convex simple pentagons that are lattice tangrams.
\end{proposition}

%%%%%%%%%%%%%%%%%%

\subsection{Non-convex non-lattice pentagons}
Now we consider non-convex pentagonal tangrams $T$ that are in the situation of Lemma~\ref{lem}(II). Then the straight line spanned by one of the sides emanating from the non-convex vertex dissects $T$ into two generalized tangrams $T_1$ and $T_2$. The endpoint of the dissecting segment is a joint vertex of $T_1$ and $T_2$, and it is their only joint vertex, since the lattices associated to the vertices of $T_1$ and $T_2$ share at most one point. Hence the number of vertices of $T$ is the sum of those of $T_1$ and $T_2$ minus $1$ if that joint vertex is a vertex of $T$, too, or that sum minus $2$ if the joint vertex is in the relative interior of a side of $T$. We obtain the following.

\begin{fact}
W.l.o.g., each side of any tan in $T$ has one of the directions $(1,0)$, $(1,1)$, $(0,1)$ or $(-1,1)$. The pentagon $T$
(together with its dissection into tans) splits into two convex generalized tangrams $T_1$ and $T_2$. The vertices of $T_1$ belong to the lattice $\mathbb{Z}^2$, the vertices of $T_2$ belong to an image of $\mathbb{Z}^2$ under a rotation by $\frac{\pi}{4}$. The polygons $T_1$ and $T_2$ have exactly one vertex $v_0$ in common. If their angles at $v_0$ add up to $\pi$ then $T_1$ is a quadrangle and $T_2$ a triangle. Otherwise the sum of that angles is smaller than $\pi$ and both $T_1$ and $T_2$ are triangles.
\end{fact}

The area of the triangle $T_2$ is larger than $0$ and smaller than $8$. All possible generalized tangrams with that property are illustrated in Figure~\ref{fig:parts3}. 
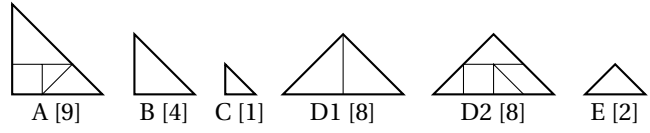
\begin{figure}
\begin{center}
\begin{tikzpicture}[xscale=.4, yscale=.4]
%A
\draw
  (0,0) node[above right] {\tikz[xscale=.4, yscale=.4]
	{
\draw[thick]
  (0,0)--(3,0)--(0,3)--cycle
	(1.5,0) node[below] {A [9]}
	;
\draw
  (0,1)--(2,1)--(1,0)--(1,1)
  ;
	}}
	;
%B
\draw
  (4,0) node[above right] {\tikz[xscale=.4, yscale=.4]
	{
\draw[thick]
  (0,0)--(2,0)--(0,2)--cycle
	(1,0) node[below] {B [4]}
	;
\draw
  ;
	}}
	;
%C
\draw
  (6.5,0) node[above right] {\tikz[xscale=.4, yscale=.4]
	{
\draw[thick]
  (0,0)--(1,0)--(0,1)--cycle
	(.5,0) node[below] {C [1]}
	;
\draw
  ;
	}}
	;
%D1
\draw
  (9,0) node[above right] {\tikz[xscale=.4, yscale=.4]
	{
\draw[thick]
  (0,0)--(4,0)--(2,2)--cycle
	(2,0) node[below] {D1 [8]}
	;
\draw
  (2,0)--(2,2)
  ;
	}}
	;
%D2
\draw
  (14,0) node[above right] {\tikz[xscale=.4, yscale=.4]
	{
\draw[thick]
  (0,0)--(4,0)--(2,2)--cycle
	(2,0) node[below] {D2 [8]}
	;
\draw
  (3,0)--(2,1)--(2,0) (1,0)--(1,1)--(3,1)
  ;
	}}
	;
%E
\draw
  (19,0) node[above right] {\tikz[xscale=.4, yscale=.4]
	{
\draw[thick]
  (0,0)--(2,0)--(1,1)--cycle
	(1,0) node[below] {E [2]}
	;
\draw
  ;
	}}
	;
\end{tikzpicture}
\end{center}
\caption{All triangular generalized lattice tangrams of area less than $8$.
\label{fig:parts3}}
\end{figure}
There they are arranged within the lattice $\mathbb{Z}^2$. In square brackets we gives the number of basic triangles that form the respective polygon, i.e., twice the area of that polygon. Figure~\ref{fig:parts3} depicts only those dissections into tans that will be used in the sequel.

Since the areas of $T_1$ and $T_2$ sum up to $8$, the area of $T_1$ must be one of $\frac{7}{2}$, $4$, $6$, $7$ or $\frac{15}{2}$. In other words, that area is represented by $7$, $8$, $12$, $14$ or $15$ basic triangles. Thus the only possible triangle representing $T_1$ can be D1 (or D2) from Figure~\ref{fig:parts3}. 

Next we obtain all quadrangular candidates for $T_1$. The sizes $k_1\frac{\pi}{4}$, $k_2\frac{\pi}{4}$, $k_3\frac{\pi}{4}$ and $k_4\frac{\pi}{4}$ of the successive inner angles of $T_2$ are given by the string $k_1k_2k_3k_4 \in \{3311,3131,3221,3212,2222\}$ w.l.o.g.\ (cf.\ the last subsection). Since the vertices and sides of $T_1$ are arranged along $\mathbb{Z}^2$, it is easy to find all possible candidates that are composed by $7$, $8$, $12$, $14$ or $15$ basic triangles. Figure~\ref{fig:parts4} depicts only those of them who can be tiled by a subfamily of all seven tans, because the others cannot represent $T_1$. Quadrangles with dotted dissections will not give rise to pentagonal tangrams in the end.
\begin{figure}
\begin{center}
\begin{tikzpicture}[xscale=.4, yscale=.4]
%F
\draw
  (-3,4) node[right] {3311:}
  (0,1) node[above right] {\tikz[xscale=.4, yscale=.4]
	{
\draw[thick]
  (0,0)--(5,0)--(4,1)--(1,1)--cycle
	(2.5,0) node[below] {F [8]}
	;
\draw
  (1,1)--(2,0)--(2,1) (3,0)--(3,1)--(4,0)
  ;
	}}
	;
%G
\draw
  (6,1) node[above right] {\tikz[xscale=.4, yscale=.4]
	{
\draw[thick]
  (0,0)--(5,0)--(3,2)--(2,2)--cycle
	(2.5,0) node[below] {G [12]}
	;
\draw
  (2,2)--(2,0)--(3,1)--(4,0)--(4,1) (3,2)--(3,1)--(2,1)
  ;
	}}
	;
%H
\draw
  (11,1) node[above right] {\tikz[xscale=.4, yscale=.4]
	{
\draw[thick]
  (0,1)--(1,0)--(3,0)--(0,3)--cycle
	(2,0) node[below] {H [8]}
	;
\draw[densely dotted]
  (0,1)--(2,1)--(2,0) (1,0)--(1,1)
  ;
	}}
	;
%I
\draw
  (13,1) node[above right] {\tikz[xscale=.4, yscale=.4]
	{
\draw[thick]
  (0,2)--(2,0)--(4,0)--(0,4)--cycle
	(3,0) node[below] {I [12]}
	;
\draw[densely dotted]
  (0,2)--(2,2)--(2,0) (2,1)--(3,1)--(3,0)
  ;
	}}
	;
%J
\draw
  (17,1) node[above right] {\tikz[xscale=.4, yscale=.4]
	{
\draw[thick]
  (0,1)--(1,0)--(4,0)--(0,4)--cycle
	(2.5,0) node[below] {J [15]}
	;
\draw
  (0,2)--(2,2)--(2,0)--(1,1)--(0,1) (1,2)--(1,1)--(2,2)
  ;
	}}
	;
%K
\draw
  (-3,-.2) node[right] {3131:}
  (0,-3.2) node[above right] {\tikz[xscale=.4, yscale=.4]
	{
\draw[thick]
  (0,0)--(4,0)--(5,1)--(1,1)--cycle
	(2,0) node[below] {K [8]}
	;
\draw
  (1,1)--(2,0)--(2,1) (3,1)--(3,0)--(4,1)
  ;
	}}
	;
%L
\draw
  (5,-3.2) node[above right] {\tikz[xscale=.4, yscale=.4]
	{
\draw[thick]
  (0,0)--(2,0)--(4,2)--(2,2)--cycle
	(1,0) node[below] {L [8]}
	;
\draw
  (2,2)--(2,0)
  ;
	}}
	;
%M
\draw
  (8,-3.2) node[above right] {\tikz[xscale=.4, yscale=.4]
	{
\draw[thick]
  (0,0)--(3,0)--(5,2)--(2,2)--cycle
	(1.5,0) node[below] {M [12]}
	;
\draw
  (2,0)--(2,2)--(3,1)--(4,2) (2,1)--(4,1) (3,0)--(3,1)
  ;
	}}
	;
%N
\draw
  (11.7,-3.2) node[above right] {\tikz[xscale=.4, yscale=.4]
	{
\draw[thick]
  (0,0)--(2,0)--(5,3)--(3,3)--cycle
	(1,0) node[below] {N [12]}
	;
\draw[densely dotted]
  (2,0)--(2,2)--(3,2) (2,1)--(3,1)--(3,3)
  ;
	}}
	;
%O
\draw
  (-3,-4.4) node[right] {3221:}
  (0,-7.4) node[above right] {\tikz[xscale=.4, yscale=.4]
	{
\draw[thick]
  (0,0)--(4,0)--(3,1)--(0,1)--cycle
	(2,0) node[below] {O [7]}
	;
\draw[densely dotted]
  (1,1)--(1,0)--(2,1) (2,0)--(3,1)
  ;
	}}
	;
%P
\draw
  (5,-7.4) node[above right] {\tikz[xscale=.4, yscale=.4]
	{
\draw[thick]
  (0,0)--(3,0)--(1,2)--(0,2)--cycle
	(1.5,0) node[below] {P [8]}
	;
\draw
  (0,0)--(1,1)--(2,0) (0,1)--(2,1) (1,1)--(1,2)
  ;
	}}
	;
%Q
\draw
  (9,-7.4) node[above right] {\tikz[xscale=.4, yscale=.4]
	{
\draw[thick]
  (0,0)--(4,0)--(2,2)--(0,2)--cycle
	(2,0) node[below] {Q [12]}
	;
\draw
  (0,0)--(2,2) (1,0)--(1,1)--(3,1) (2,0)--(2,1)--(3,0)
  ;
	}}
	;
%R
\draw
  (14,-7.4) node[above right] {\tikz[xscale=.4, yscale=.4]
	{
\draw[thick]
  (0,0)--(4,0)--(1,3)--(0,3)--cycle
	(2,0) node[below] {R [15]}
	;
\draw
  (0,0)--(2,2)--(2,0) (0,2)--(1,2)--(0,1) (1,3)--(1,1)
  ;
	}}
	;
%S
\draw
  (16,-7.4) node[above right] {\tikz[xscale=-.4, yscale=.4]
	{
\draw[thick]
  (0,0)--(2,0)--(5,3)--(4,4)--cycle
	(1,0) node[below] {S [14]}
	;
\draw[densely dotted]
  (1,0)--(1,1)--(2,1) (2,0)--(2,2)--(4,2)--(4,4)
  ;
	}}
	;
%T
\draw
  (-3,-8.8) node[right] {3212:}
  (0,-11.8) node[above right] {\tikz[xscale=.4, yscale=.4]
	{
\draw[thick]
  (0,0)--(3,0)--(1,2)--(0,1)--cycle
	(1.5,0) node[below] {T [7]}
	;
\draw
  (0,0)--(1,1) (1,0)--(1,2)
  ;
	}}
	;
%U
\draw
  (4,-11.8) node[above right] {\tikz[xscale=.4, yscale=.4]
	{
\draw[thick]
  (0,0)--(4,0)--(1,3)--(0,2)--cycle
	(2,0) node[below] {U [14]}
	;
\draw[densely dotted]
  (0,1)--(1,1)--(1,0) (0,2)--(2,2)--(2,0)--cycle
  ;
	}}
	;
%V
\draw
  (-3,-12) node[right] {2222:}
  (0,-15) node[above right] {\tikz[xscale=.4, yscale=.4]
	{
\draw[thick]
  (0,0)--(4,0)--(4,1)--(0,1)--cycle
	(2,0) node[below] {V [8]}
	;
\draw
  (0,1)--(1,0)--(2,1) (2,0)--(3,1)--(3,0)
  ;
	}}
	;
%W
\draw
  (4.8,-15) node[above right] {\tikz[xscale=.4, yscale=.4]
	{
\draw[thick]
  (0,0)--(2,0)--(2,2)--(0,2)--cycle
	(1,0) node[below] {W [8]}
	;
\draw
  (0,0)--(2,2)
  ;
	}}
	;
%X
\draw
  (8,-15) node[above right] {\tikz[xscale=.4, yscale=.4]
	{
\draw[thick]
  (0,0)--(3,0)--(3,2)--(0,2)--cycle
	(1.5,0) node[below] {X [12]}
	;
\draw[densely dotted]
  (0,2)--(1,1)--(1,0)--(3,2) (0,1)--(2,1)--(1,2)
  ;
	}}
	;
%Y
\draw
  (11.95,-15) node[above right] {\tikz[xscale=.4, yscale=.4]
	{
\draw[thick]
  (1,0)--(3,2)--(2,3)--(0,1)--cycle
	(1,0) node[below] {Y [8]}
	;
\draw[densely dotted]
  (0,1)--(2,1)--(2,3) (2,2)--(3,2)
  ;
	}}
	;
%Z
\draw
  (14.75,-15) node[above right] {\tikz[xscale=.4, yscale=.4]
	{
\draw[thick]
  (1,0)--(4,3)--(3,4)--(0,1)--cycle
	(1,0) node[below] {Z [12]}
	;
\draw[densely dotted]
  (0,1)--(2,1)--(2,3)--(4,3) (3,3)--(3,4)
  ;
	}}
	;
\end{tikzpicture}
\end{center}
\caption{All convex quadrangular generalized lattice tangrams of areas $\frac{7}{2}$, $4$, $6$, $7$ or $\frac{15}{2}$ that can be dissected into a subfamily of all seven tans.
\label{fig:parts4}}
\end{figure}
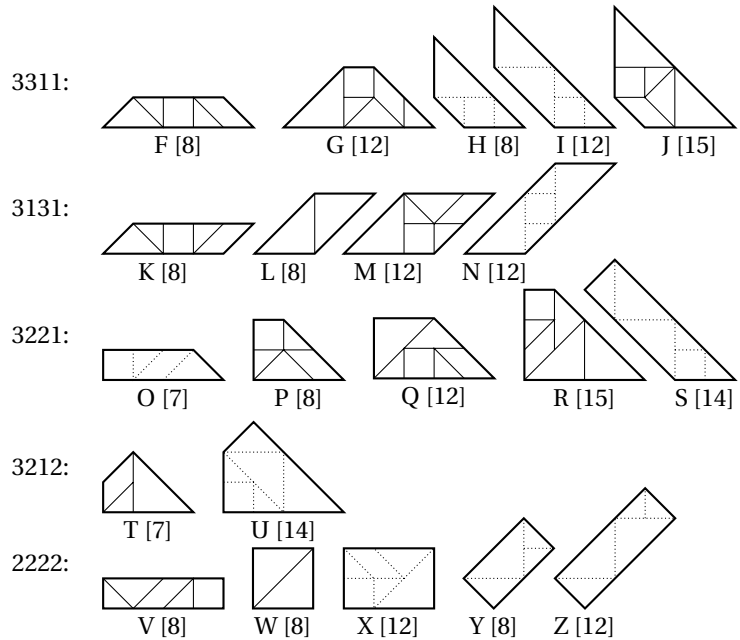

Finally, we consider every remaining candidate for $T_1$, which is either triangle D1 (or D2) from Figure~\ref{fig:parts3} or a quadrangle from Figure~\ref{fig:parts4}. We pick a triangle of the respective size from Figure~\ref{fig:parts3} as a candidate $T_2'$ for the shape of $T_2$. We check if both $T_1$ and $T_2'$ can be tiled simultaneously by all seven tans. This excludes the quadrangles H, I, N, O, S, U, Y and Z. If the tiling is possible, we choose all images $T_2$ of $T_2'$ under a rotation by an odd multiple of $\frac{\pi}{4}$ such that $T_1$ together with $T_2$ forms a dissection of a pentagon $T$. The last is impossible for the quadrangle X. This way we obtain the $31$ mutually incongruent pentagons $T$ from Figure~\ref{fig:nonlattice}. Respective dissections can be adopted from Figures~\ref{fig:parts3} and~\ref{fig:parts4}. The dissections are not unique in general.
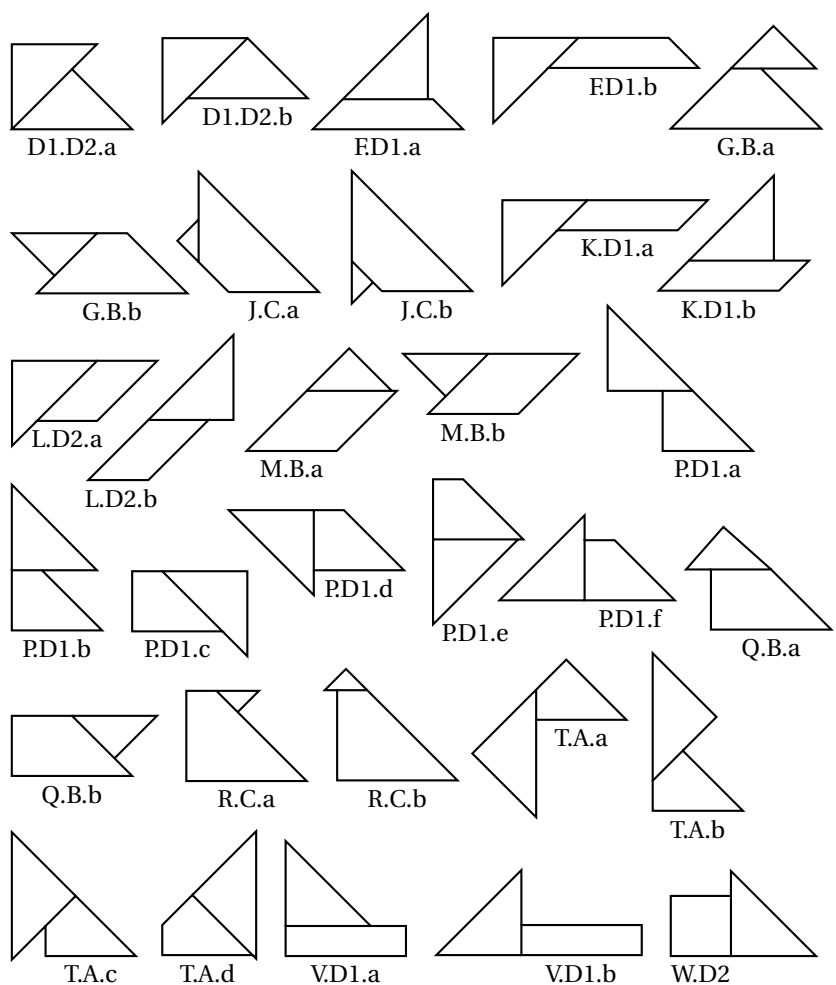
\begin{figure}
\begin{center}
\begin{tikzpicture}[xscale=.4, yscale=.4]
%D1a
\draw
  (0,-.3) node[above right] {\tikz[xscale=.4, yscale=.4]
	{
\draw[thick]
  (0,0)--(0,2.828)--(2.828,2.828)--(2,2)
  (0,0)--(4,0)--(2,2)--cycle
	(2,0) node[below] {D1.D2.a}
	;
	}}
	;
%D1b
\draw
  (5,.7) node[above right] {\tikz[xscale=.4, yscale=.4]
	{
\draw[thick]
  (0,0)--(-.828,-.828)--(-.828,2)--(2,2)
  (0,0)--(4,0)--(2,2)--cycle
	(2,0) node[below] {D1.D2.b}
	;
  ;
	}}
	;
%Fa
\draw
  (10,-.3) node[above right] {\tikz[xscale=.4, yscale=.4]
	{
\draw[thick]
  (3.828,1)--(3.828,3.828)--(1,1)
  (0,0)--(5,0)--(4,1)--(1,1)--cycle
	(2.5,0) node[below] {F.D1.a}
	;
	}}
	;
%Fb
\draw
  (16,1) node[above right] {\tikz[xscale=.4, yscale=.4]
	{
\draw[thick]
  (1,1)--(-1.828,1)--(-1.828,-1.828)--(0,0)
  (0,0)--(5,0)--(4,1)--(1,1)--cycle
	(2.5,0) node[below] {F.D1.b}
	;
	}}
	;
%Ga
\draw
  (21.9,-.3) node[above right] {\tikz[xscale=.4, yscale=.4]
	{
\draw[thick]
  (3,2)--(4.828,2)--(3.414,3.414)--(2,2)
  (0,0)--(5,0)--(3,2)--(2,2)--cycle
	(2.5,0) node[below] {G.B.a}
	;
	}}
	;
%Gb
\draw
  (0,-5.8) node[above right] {\tikz[xscale=.4, yscale=.4]
	{
\draw[thick]
  (2,2)--(-.828,2)--(.586,.586)
  (0,0)--(5,0)--(3,2)--(2,2)--cycle
	(2.5,0) node[below] {G.B.b}
	;
	}}
	;
%Ja
\draw
  (5.5,-5.8) node[above right] {\tikz[xscale=.4, yscale=.4]
	{
\draw[thick]
  (0,1)--(-.707,1.707)--(0,2.414)
  (0,1)--(1,0)--(4,0)--(0,4)--cycle
	(2.5,0) node[below] {J.C.a}
	;
	}}
	;
%Jb
\draw
  (11.3,-5.8) node[above right] {\tikz[xscale=.4, yscale=.4]
	{
\draw[thick]
  (0,1)--(0,-.414)--(.707,.293)
  (0,1)--(1,0)--(4,0)--(0,4)--cycle
	(2.5,0) node[below] {J.C.b}
	;
	}}
	;
%Ka
\draw
  (16.3,-4.4) node[above right] {\tikz[xscale=.4, yscale=.4]
	{
\draw[thick]
  (0,0)--(-1.828,-1.828)--(-1.828,1)--(1,1)
  (0,0)--(4,0)--(5,1)--(1,1)--cycle
	(2,0) node[below] {K.D1.a}
	;
	}}
	;
%Kb
\draw
  (21.5,-5.7) node[above right] {\tikz[xscale=.4, yscale=.4]
	{
\draw[thick]
  (1,1)--(3.828,3.828)--(3.828,1)
  (0,0)--(4,0)--(5,1)--(1,1)--cycle
	(2,0) node[below] {K.D1.b}
	;
	}}
	;
%La
\draw
  (0,-10) node[above right] {\tikz[xscale=.4, yscale=.4]
	{
\draw[thick]
  (0,0)--(-.828,-.828)--(-.828,2)--(2,2)
  (0,0)--(2,0)--(4,2)--(2,2)--cycle
	(1,0) node[below] {L.D2.a}
	;
	}}
	;
%Lb
\draw
  (2.2,-12) node[above right] {\tikz[xscale=.4, yscale=.4]
	{
\draw[thick]
  (2,2)--(4.828,4.828)--(4.828,2)--(4,2)
  (0,0)--(2,0)--(4,2)--(2,2)--cycle
	(1.1,0) node[below] {L.D2.b}
	;
	}}
	;
%Ma
\draw
  (7.8,-11) node[above right] {\tikz[xscale=.4, yscale=.4]
	{
\draw[thick]
  (2,2)--(3.414,3.414)--(4.828,2)
  (0,0)--(3,0)--(5,2)--(2,2)--cycle
	(1.5,0) node[below] {M.B.a}
	;
	}}
	;
%Mb
\draw
  (13,-9.8) node[above right] {\tikz[xscale=.4, yscale=.4]
	{
\draw[thick]
  (2,2)--(-.828,2)--(.586,.586)
  (0,0)--(3,0)--(5,2)--(2,2)--cycle
	(1.5,0) node[below] {M.B.b}
	;
	}}
	;
%Pa
\draw
  (19.8,-11) node[above right] {\tikz[xscale=.4, yscale=.4]
	{
\draw[thick]
  (1,2)--(-1.828,4.828)--(-1.828,2)--(0,2)
  (0,0)--(3,0)--(1,2)--(0,2)--cycle
	(1.5,0) node[below] {P.D1.a}
	;
	}}
	;
%Pb
\draw
  (0,-17) node[above right] {\tikz[xscale=.4, yscale=.4]
	{
\draw[thick]
  (1,2)--(2.828,2)--(0,4.848)--(0,2)
  (0,0)--(3,0)--(1,2)--(0,2)--cycle
	(1.5,0) node[below] {P.D1.b}
	;
	}}
	;
%Pc
\draw
  (4,-17) node[above right] {\tikz[xscale=.4, yscale=.4]
	{
\draw[thick]
  (1,2)--(3.828,2)--(3.828,-.828)--(3,0)
  (0,0)--(3,0)--(1,2)--(0,2)--cycle
	(1.5,0) node[below] {P.D1.c}
	;
	}}
	;
%Pd
\draw
  (7.2,-15) node[above right] {\tikz[xscale=.4, yscale=.4]
	{
\draw[thick]
  (0,2)--(-2.828,2)--(0,-.828)--(0,0)
  (0,0)--(3,0)--(1,2)--(0,2)--cycle
	(1.5,0) node[below] {P.D1.d}
	;
	}}
	;
%Pe
\draw
  (14,-16.45) node[above right] {\tikz[xscale=.4, yscale=.4]
	{
\draw[thick]
  (0,0)--(0,-2.828)--(2.828,0)
  (0,0)--(3,0)--(1,2)--(0,2)--cycle
	(1.4,-2.5) node[below] {P.D1.e}
	;
	}}
	;
%Pf
\draw
  (16.2,-16) node[above right] {\tikz[xscale=.4, yscale=.4]
	{
\draw[thick]
  (0,0)--(-2.828,0)--(0,2.828)--(0,2)
  (0,0)--(3,0)--(1,2)--(0,2)--cycle
	(1.5,0) node[below] {P.D1.f}
	;
	}}
	;
%Qa
\draw
  (22.4,-17.1) node[above right] {\tikz[xscale=.4, yscale=.4]
	{
\draw[thick]
  (0,2)--(-.828,2)--(.414,3.414)--(2,2)
  (0,0)--(4,0)--(2,2)--(0,2)--cycle
	(2,0) node[below] {Q.B.a}
	;
	}}
	;
%Qb
\draw
  (0,-22) node[above right] {\tikz[xscale=.4, yscale=.4]
	{
\draw[thick]
  (2,2)--(4.828,2)--(3.414,.586)
  (0,0)--(4,0)--(2,2)--(0,2)--cycle
	(2,0) node[below] {Q.B.b}
	;
	}}
	;
%Ra
\draw
  (5.8,-22) node[above right] {\tikz[xscale=.4, yscale=.4]
	{
\draw[thick]
  (1,3)--(2.414,3)--(1.707,2.293)
  (0,0)--(4,0)--(1,3)--(0,3)--cycle
	(2,0) node[below] {R.C.a}
	;
	}}
	;
%Rb
\draw
  (10.4,-22) node[above right] {\tikz[xscale=.4, yscale=.4]
	{
\draw[thick]
  (0,3)--(-.414,3)--(.293,3.707)--(1,3)
  (0,0)--(4,0)--(1,3)--(0,3)--cycle
	(2,0) node[below] {R.C.b}
	;
	}}
	;
%Ta
\draw
  (15.3,-22.1) node[above right] {\tikz[xscale=.4, yscale=.4]
	{
\draw[thick]
  (0,1)--(-2.121,-1.121)--(0,-3.242)--(0,0)
  (0,0)--(3,0)--(1,2)--(0,1)--cycle
	(1.5,0) node[below] {T.A.a}
	;
	}}
	;
%Tb
\draw
  (21.3,-23) node[above right] {\tikz[xscale=.4, yscale=.4]
	{
\draw[thick]
  (0,1)--(0,5.242)--(2.121,3.121)--(1,2)
	(0,0)--(3,0)--(1,2)--(0,1)--cycle
	(1.5,0) node[below] {T.A.b}
	;
	}}
	;
%Tc
\draw
  (0,-27.8) node[above right] {\tikz[xscale=.4, yscale=.4]
	{
\draw[thick]
  (0,1)--(-1.121,-.121)--(-1.121,4.121)--(1,2)
  (0,0)--(3,0)--(1,2)--(0,1)--cycle
	(1.5,0) node[below] {T.A.c}
	;
	}}
	;
%Td
\draw
  (5,-27.8) node[above right] {\tikz[xscale=.4, yscale=.4]
	{
\draw[thick]
  (1,2)--(3.121,4.121)--(3.121,-.121)--(3,0)
  (0,0)--(3,0)--(1,2)--(0,1)--cycle
	(1.5,0) node[below] {T.A.d}
	;
	}}
	;
%Va
\draw
  (9.1,-27.8) node[above right] {\tikz[xscale=.4, yscale=.4]
	{
\draw[thick]
  (0,1)--(0,3.828)--(2.828,1)
  (0,0)--(4,0)--(4,1)--(0,1)--cycle
	(2,0) node[below] {V.D1.a}
	;
	}}
	;
%Vb
\draw
  (14.1,-27.8) node[above right] {\tikz[xscale=.4, yscale=.4]
	{
\draw[thick]
  (0,0)--(-2.828,0)--(0,2.828)--(0,1)
  (0,0)--(4,0)--(4,1)--(0,1)--cycle
	(2,0) node[below] {V.D1.b}
	;
	}}
	;
%W
\draw
  (21.7,-27.8) node[above right] {\tikz[xscale=.4, yscale=.4]
	{
\draw[thick]
  (2,2)--(2,2.828)--(4.828,0)--(2,0)
  (0,0)--(2,0)--(2,2)--(0,2)--cycle
	(1,0) node[below] {W.D2}
	;
	}}
	;
\end{tikzpicture}
\end{center}
\caption{All $31$ non-convex pentagonal non-lattice tangrams.
\label{fig:nonlattice}}
\end{figure}

\begin{proposition}
There exist, up to isometry, exactly $31$ non-convex simple pentagons that are non-lattice tangrams.
\end{proposition}

The characterization of all simple pentagonal tangrams is complete.

%%%%%%%%%%%%%%%%%%%%%%%%%%%%%%%%%%%%%%%%%%%%%%%%%

\section*{Acknowledgments}

Both authors express their gratitude to Dr.\ Carsten M\"uller for sharing his enthusiasm, posing problems and encouraging work on tangrams.

%%%%%%%%%%%%%%%%%%%%%%%%%%%%%%%%%%%%%%%%%%%%%%%%%

\bibliographystyle{plain}

%%%%%%%%%%%%%%%%%%%%%%%%%%%%%%%%%%%%%%%%%%%%%%%%%%%%%%%%%%%%%%

\end{document}